\documentclass[11pt]{amsart}
\usepackage[margin=1.25in]{geometry}

\usepackage{amsmath, amscd}
\usepackage{mathrsfs}
\usepackage{amsfonts}
\usepackage{latexsym,epsfig}
\usepackage{mathabx}
\usepackage{amsmath, amscd, amsthm}
\usepackage[pdftex]{color}
\usepackage{comment}
\usepackage{marginnote}

\newcommand{\MM}[0]{\mathcal{M}}

\newcommand{\CC}[0]{\mathcal{C}}

\usepackage{hyperref}

\newtheorem{theorem}{Theorem}[section]
\newtheorem{lemma}[theorem]{Lemma}
\newtheorem{proposition}[theorem]{Proposition}
\newtheorem{corollary}[theorem]{Corollary}

\newtheorem{definition}{Definition}

\theoremstyle{definition}
\newtheorem{remark}[theorem]{Remark}

\newcommand{\T}{\mathcal{T}}

\newcommand{\M}{\mathcal{M}}
\newcommand{\C}{\mathcal{C}}
\newcommand{\Q}{\mathcal{Q}}
\newcommand{\Mod}{\mathrm{Mod}}
\newcommand{\pr}{\mathrm{Proj}}
\newcommand{\hpr}{\mathrm{proj}}

\begin{document}

\title{Convex cocompactness and stability in mapping class groups}
\author[M.G. Durham]{Matthew Gentry Durham}
\address{University of Illinois at Chicago,
851 S Morgan St,
Chicago, IL 60607, U.S.A.
}
\email{\href{mailto:mdurha2@uic.edu}{mdurha2@uic.edu}}
\author[S.J. Taylor]{Samuel J. Taylor}
\address{Department of Mathematics, 
University of Texas at Austin, 
1 University Station C1200, 
Austin, TX 78712, U.S.A.
}
\email{\href{mailto:staylor@math.utexas.edu}{staylor@math.utexas.edu}}
\date{\today}
\maketitle

\begin{abstract}
We introduce a strong notion of quasiconvexity in finitely generated groups, which we call \emph{stability}. Stability agrees with quasiconvexity in hyperbolic groups and is preserved under quasi-isometry for finitely generated groups. We show that the stable subgroups of mapping class groups are precisely the convex cocompact subgroups. This generalizes a well-known result of Behrstock \cite{Be} and is related to questions asked by Farb-Mosher \cite{FMo} and Farb \cite{Farbproblems}.
\end{abstract}

\section{Introduction}
In order to understand the structure of a finitely generated group $G$, one often investigates subgroups $H \le G$ whose geometry reflects that of $G$. One successful application of this approach is to the study of quasiconvex subgroups of hyperbolic groups.  In this setting, $H$ is finitely generated and undistorted in $G$ and these properties are preserved under quasi-isometries of $G$.  Quasiconvexity, however, is not as useful for arbitrary finitely generated groups. Without hyperbolicity of $G$, quasiconvexity depends on a choice of generating set for $G$ and, in particular, is not preserved under quasi-isometry. To address this situation, we introduce the stronger notion of \emph{stability}, which agrees with quasiconvexity when $G$ is hyperbolic. Specifically, we define the following:

\begin{definition}
Let $G$ be a finitely generated group. A subgroup $H \le G$ is \emph{stable} if $H$ is undistorted in $G$ and for all $L \ge 0$ there exists an $R= R(L) \ge 0$ satisfying the following: for any pair of $L$-quasigeodesics of $G$ that share common endpoints in $H$, each is contained in the $R$-neighborhood of the other.
\end{definition}

Our primary motivation for defining stable subgroups of a finitely generated group is the mapping class group $\Mod(S)$ of a connected, orientable surface $S$. In this note, we relate stable subgroups of $\Mod(S)$ to \emph{convex cocompact subgroups} of the mapping class group, introduced by Farb and Mosher in \cite{FMo}. These are much studied subgroups of $\Mod(S)$ that have important connections to the geometry of Teichm\"{u}ller space, the curve graph, and surface group extensions. (See Section \ref{Modbackground} for definitions.) Our main result can be interpreted as a generalization of a theorem of Behrstock \cite{Be} (also see \cite{DMS}). Behrstock proves that the stable (or Morse) elements of $\Mod(S)$ are exactly the pseudo-Anosov mapping classes.

Our main theorem provides a purely group theoretic characterization of convex compactness, which does not involve the geometry of either Teichm\"{u}ller space or the curve graph. This distinguishes our characterization of convex cocompact subgroups of mapping class groups from those appearing in \cite{FMo,KL2,H}. We prove

\begin{theorem}\label{intro_main}
The subgroup $G \le \Mod(S)$ is stable if and only if it is convex cocompact.
\end{theorem}

Theorem \ref{intro_main} partially answers questions appearing in \cite{FMo} and \cite{Farbproblems}. In particular, Farb and Mosher ask how their notion of convex cocompactness (which they define as having quasiconvex orbit in Teichm\"{u}ller space) is related to quasiconvexity in the mapping class group \cite{FMo}. Also, in Problem $3.8$ of \cite{Farbproblems}, Farb asks what subgroups of mapping class groups are quasiconvex with respect to fixed generating sets. Theorem \ref{intro_main} characterizes the subgroups of the mapping class group that satisfy our strong notion of quasiconvexity and implies that convex cocompact subgroups are quasiconvex in $\Mod(S)$ with respect to any generating set (Proposition \ref{anygen}). It is our hope that this notion of stability will also be useful in other finitely generated groups.\\

\noindent \textbf{Acknowledgements} \, We thank Daniel Groves, Alan Reid, and Jing Tao for helpful conversations. We are also grateful for the hospitality of both the MSRI during its Hot Topics: Surface subgroups and cube complexes workshop and the MRC program during its geometric groups theory workshop in 2013.

\section{Background}
We keep this section brief and refer the reader to \cite{BH} for background on coarse geometry and \cite{FM} for mapping class group basics. Throughout, we assume that the reader has some familiarity with subsurface projections and hierarchies for the mapping class group, as introduced in \cite{MM2}. See \cite{Mi, Be} for additional references.

\subsection{Coarse geometry}
Let $(X,d_X)$ and $(Y,d_Y)$ be metric spaces.  Recall that $f: X \to Y$ is a $K$-\emph{quasi-isometric embedding} if for all $x_1,x_2 \in X$,
$$\frac{1}{K} d_X(x_1,x_2) - K \le d_Y(f(x_1), f(x_2)) \le Kd_X(x_1,x_2) +K. $$

We remark that what we have defined is usually called a $(K,K)$-quasi-isometric embedding in the literature, but our definition will reduce the number of constants appearing throughout this note. If $f:X \to Y$ has the additional property that every point in $Y$ is within $K$ of the image $f(X)$, then $f$ is a $K$-quasi-isometry and $X$ and $Y$ are \emph{quasi-isometric}.

If $I$ is a subinterval of either $\mathbb{R}$ or $\mathbb{Z}$, then a $K$-quasi-isometric embedding $f: I \to X$ is called a $K$-\emph{quasigeodesic}. We will often refer to $f$ as a quasigeodesic and call $K$ the \emph{quasigeodesic constant} for $f$. When $f$ is an isometric embedding, it is called a \emph{geodesic}. The metric space $X$ is called \emph{geodesic} if for any $x_1,x_2 \in X$ there is a geodesic $f:[0,N] \to X$ with $f(0) = x_1$ and $f(N) = x_2$, i.e. there is a geodesic joining $x_1$ to $x_2$. We will sometimes write $[x_1,x_2]$ to denote an arbitrary geodesic joining $x_1$ and $x_2$. For any path $\gamma: I \to X$, we will continue to use the symbol $\gamma$ to denote the image of $\gamma$ in $X$, as what is meant will be clear from context.

Recall that a subset $C$ of a geodesic metric space $X$ is $K$ -\emph{quasiconvex} if for any $c_1,c_2 \in C$ and any geodesic $[c_1,c_2]$ in $X$,  $[c_1,c_2] \subset N_K(C)$. Here, $N_K(C)$ denotes the closed $K$-neighborhood of $C$. For any $\epsilon>0$, two subsets $A$ and $B$ of $X$ have \emph{Hausdorff distance no greater than} $\epsilon$ if $A \subset N_{\epsilon}(B)$ and $B \subset N_{\epsilon}(A)$. The infimum over all such $\epsilon$ is the \emph{Hausdorff distance} between $A$ and $B$, denoted by $d_{\mathrm{Haus(X)}}(A,B)$. Throughout this note, we reserve the notation $d(A,B)$ to denote the diameter of the union of $A$ and $B$. In symbols, $d(A,B) = \mathrm{diam}(A \cup B)$.

We make one further remark on notation. The expression $A \prec B$ is defined to mean that there exists a $K \ge 1$ so that $A \le K \cdot B +K$. In different contexts the constant $K$ will depend on particular parameters but not on the numbers $A$ and $B$ directly. We define $A \succ B$ similarly and write $A \asymp B$ if both $A \prec B$ and $B \prec A$. When using this notation below, we will be clear about the dependence of $K$.

\subsection{Hyperbolic geometry}\label{hypgeo}
A geodesic metric space $X$ is \emph{$\delta$-hyperbolic} if for any $x,y,z \in X$ and any geodesic segments $[x,y],[y,z],[x,z]$ joining them,  $[x,z] \subset N_{\delta}([x,y]\cup [y,z])$. That is, $X$ has $\delta$-thin triangles. In this note, we will need a few well-known properties about the nearest point retraction from a hyperbolic metric space $X$ to a quasigeodesic $\gamma$ in $X$. See \cite{BH} for additional details.

Let $\gamma:[0,N] \to X$ be a $K$-quasigeodesic. The \emph{nearest point retraction} from $X$ to $\gamma$ is a map $\bold{n} = \bold{n}_{\gamma}: X \to \mathrm{im}(\gamma)$ defined as follows: for $x\in X$, $\bold{n}(x)$ is any point in the image of $\gamma$ such that $d_X(x,\bold{n}(x)) = \min_{i \in [0,N]}d_X(x, \gamma(i))$. In the case that $X$ is $\delta$-hyperbolic, there is a $p \ge 0$ depending only on $K$ and $\delta$ such that if $\gamma(j)$ is a different point on $\gamma$ minimizing distance to $x$, then $d_X(\bold{n}(x),\gamma(j)) \le p$. Moreover, $\bold{n}(\gamma(i)) = \gamma(i)$ for any $i \in [0,N]$ and for any $x,y \in X$,
$$d_X(\bold{n}(x),\bold{n}(y)) \le p \cdot d(x,y) +p. $$

In Section \ref{ccimplies}, the nearest point retraction will be used to define a projection from the space $X$ to the domain interval of a quasigeodesic.

\subsection{Curves, markings, and hierarchy paths}\label{Modbackground}

In this section, we recall the work of Masur-Minsky on the curve and marking graphs. Fix a orientable surface $S$ with genus $g\ge0$ and $p\ge0$ punctures so that $\omega(S) = 3g-3+p \ge 1$; $\omega(S)$ is called the \emph{complexity} of $S$. The \emph{curve graph}, denoted $\C(S)$, is a locally infinite simplicial graph whose vertices are isotopy classes of essential simple closed curves on $S$ and two (isotopy classes of) curves are joined by an edge if they can be realized disjointly on $S$. The curve graph is the $1$-skeleton of a simplicial complex introduced by Harvey in \cite{Ha}.

\begin{remark}
The above definition is for $S$ with $\omega(S)\ge 2$.  When $\omega(S)=1$, the definition is modified so that $\C(S)$ is the Farey graph.  See Subsection 2.4 of \cite{MM2} for when $S$ is an annulus, i.e. $\omega(S)=-1$.
\end{remark}

Endow $\C(S)$ with the graph metric.  We frequently use the following foundational result of Masur-Minsky \cite{MM1}: 

\begin{theorem}[\cite{MM1}]
For any $S$, there is a $\delta>0$ so that $\C(S)$ is $\delta$-hyperbolic. 
\end{theorem}

A (complete clean) \emph{marking}, $\mu$, on $S$ is a pants decomposition called the \emph{base of} $\mu$, $\mathrm{base}(\mu)$, and, for each $\alpha \in \mathrm{base}(\mu)$, a \emph{transversal} $t_{\alpha}\in \C(S)$ which intersects $\alpha$ and no other base curve. The \emph{marking graph}, $\M(S)$, is a simplicial graph whose vertices are markings, with two markings connected by an edge if they differ by a Dehn (half) twist around a base curve $(\alpha, t_{\alpha}) \mapsto (\alpha, T_{\alpha}\cdot t_{\alpha})$ called a \emph{twist move}, or a \emph{flip move}, which switches a base curve and its transversal, $(\alpha, t_{\alpha}) \mapsto (t_{\alpha}, \alpha)$ (see Section 2.5 of \cite{MM2} for more details). Masur and Minsky show:

\begin{theorem}[\cite{MM2}]
$\Mod(S)$ is $\Mod(S)$-equivariantly quasi-isometric to $\M(S)$.
\end{theorem}

Often we want to compare two curves or markings on a subsurface.  For any curve $\alpha \in \C(S)$ and nonannular subsurface $Y \subset S$, the \emph{subsurface projection of $\alpha$ to $Y$} is the subset $\pi_Y(\alpha) \subset \C(Y)$ obtained by restricting $\alpha$ to $Y$ and completing the resulting arcs to curves along $\partial Y$ in a natural way (see Section 2.3 of \cite{MM2} for more details and the definition when $Y$ is an annulus). In the case of a marking $\mu \in \M(S)$, one projects only the base, that is $\pi_Y(\mu) = \pi_{Y}(\textrm{base}(\mu))$.  For $\mu_1, \mu_2 \in \M(S)$, we write $d_{Y}(\mu_1, \mu_2) = diam_{\C(Y)}(\pi_Y(\mu_1) \cup \pi_Y(\mu_2))$.

One of the main constructions from \cite{MM2} is the hierarchy machinery, from which we need only a few features of the induced hierarchy paths (see Section 4 of \cite{MM2}). Given two markings $\mu_1, \mu_2 \in \M(S)$, a \emph{hierarchy} $J$ between $\mu_1$ and $\mu_2$ is a collections of geodesics in various subsurface curve graphs whose interrelations encode the combinatorial relationship between $\mu_1$ and $\mu_2$. 

For any $A\ge0$ we call a subsurface $Y \subseteq S$ (possible $Y=S$) an $A$-\emph{large link} for $\mu_1$ and $\mu_2$ if $d_Y(\mu_1, \mu_2) \ge A$.  The following theorem says that the distance between any two markings is coarsely determined by the large links between them:

\begin{theorem}[The distance formula; Theorem 6.12 in \cite{MM2}]\label{distance}
There is a constant $A_0 \ge 0$, depending only on $S$, so that for all $A \ge A_0$ there exists $K\ge 1$ such that for any pair of markings $\mu_1, \mu_2 \in \M(S)$, we have
\[\frac{1}{K} \cdot d_{\M(S)}(\mu_1, \mu_2) - K \le \sum_{Y \subseteq S} \left[d_Y(\mu_1, \mu_2)\right]_A \le  K \cdot d_{\M(S)}(\mu_1, \mu_2)+K\]
where $[X]_A = X$ if $X\ge A$ and $0$ otherwise.
\end{theorem}

By contrast, we say that two markings $\mu_1, \mu_2 \in \M(S)$ are \emph{$E$-cobounded} if $d_Y(\mu_1, \mu_2) \le E$ for every proper subsurface $Y \subsetneqq S$.  More generally, we say that a collection of markings $M \subset \M(S)$ is $E$-cobounded if every pair of markings in $M$ is $E$-cobounded.  Coboundedness is a strong condition and paths between cobounded markings have hyperbolic behavior, a central idea in Section \ref{ccimplies}.

Though hierarchies are technical objects with many applications, for this note their utility lies in their ability to be built into \emph{hierarchy paths}. We collect some properties of hierarchy paths in the following theorem:

\begin{theorem}\label{hier qg}
There are $M, M_1, M_2 \ge0$ depending only on $S$, such that for any $\mu_1, \mu_2 \in \M(S)$, the following hold:
\begin{enumerate}
\item There is a hierarchy path $H:[0,N] \to \M(S)$ with $H(0) = \mu_1$ and $H(N)= \mu_2$, and every hierarchy path is an $M$-quasigeodesic.\label{hier qg  1}
\item For each $Y \subseteq S$, the projection of the hierarchy path $H$ to $\C(Y)$ via subsurface projection is an unparameterized quasigeodesic with uniform constants. \label{hier qg 2}
\item If $d_Y(\mu_1, \mu_2) \ge M_1$, then the set of markings in $H$ whose bases contain $\partial Y$ is a contiguous subpath denoted $H_Y$. Further, if $\alpha_Y$ and $\beta_Y$ denote the initial and terminal markings of $H_Y$, respectively, then
$$d_Y(\alpha_Y,\beta_Y) \ge d_Y(\mu_1,\mu_2) - 2M_2. $$
\label{hier qg  3}
\item For any $E>0$ there is an $E'>0$ depending only on $E$ and $S$ such that if $\mu_1, \mu_2$ are $E$-cobounded, $H$ is a hierarchy path between them, and $\mu_1', \mu_2' \in H$, then $\mu_1'$ and $\mu_2'$ are $E'$-cobounded. \label{hier qg  4}
\end{enumerate}
\end{theorem}

\begin{remark}
The above theorem essentially follows from the work in \cite{MM2}, with (\ref{hier qg 1}) being [Theorem 6.10, \cite{MM2}], (\ref{hier qg 2}) following from the construction, and (\ref{hier qg 3}) a consequence of [Lemma 5.16, \cite{Mi}].   Part (\ref{hier qg 4}) follows from (\ref{hier qg 2}) and (\ref{hier qg 3}). These statements also appear in \cite{BMMII}.
\end{remark}

\begin{remark}[Hierarchy paths between cobounded markings]
If two markings $\mu_1, \mu_2 \in \M(S)$ are $E$-cobounded, then Theorem \ref{distance} implies that $d_{\MM(S)}(\mu_1, \mu_2) \asymp d_{\CC(S)}(\mu_1, \mu_2)$.  It follows then from Theorem \ref{hier qg} (\ref{hier qg 2}) that the projection to $\CC(S)$ of any hierarchy path between $\mu_1$ and $\mu_2$ is a genuine quasigeodesic.  See Section \ref{ccimplies} below.
\end{remark}

\subsection{Convex cocompactness in $\Mod(S)$}
Convex cocompact subgroups of mapping class groups were introduced by Farb and Mosher in \cite{FMo}. A finitely generated $G \le \Mod(S)$ is \emph{convex cocompact} if for some $x \in \T(S)$ the orbit $G \cdot x$ is quasiconvex with respect to the Teichm\"{u}ller metric on $\T(S)$. Farb-Mosher verify that convex cocompactness is independent of the chosen $x \in \T(S)$ and relate convex cocompact subgroups of mapping class groups to hyperbolic extensions of surfaces groups. See \cite{FMo, H} for details.

Kent-Leininger \cite{KL2} and, independently, Hamenst\"{a}dt \cite{H} gave a characterization of convex cocompactness in terms of the curve graph $\C(S)$: 

\begin{theorem}[Kent-Leininger, Hamenst\"{a}dt] \label{cc}
Let $G \le \Mod(S)$ be finitely generated. Then $G$ is convex cocompact if and only if some (any) orbit map $G \to \C(S)$ is a quasi-isometric embedding.
\end{theorem}

Our main goal in this note is to provide a characterization of convex cocompactness in $\Mod(S)$ that uses only the geometry of $\Mod(S)$ itself, and neither Teichm\"{u}ller space nor the curve graph. This geometric characterization leads us to define the notion of \emph{stability}, which is defined for arbitrary finitely generated groups.

\section{Stability}
In this section, we define stability and provide some basic properties. Informally, a quasi-isometrically embedded subspace is stable if all quasigeodesics beginning and ending in the space are forced to fellow travel. This strong notion of convexity forces hyperbolic-like behavior around the subspace.

\begin{definition} \label{def: stab}
Let $f: Y \to X$ be a quasi-isometric embedding between geodesic metric spaces. We say $Y$ is \emph{stable} in $X$ if for any $L \ge0 $ there is a $R= R(L) \ge 0$ so that if $\gamma: [a,b] \to X$ and $\gamma': [a',b'] \to X$ are $L$-quasigeodesics with $\gamma(a) = \gamma'(a') \in f(Y)$ and $\gamma(b)= \gamma'(b') \in f(Y)$, then
$$d_{\mathrm{Haus}}(\gamma, \gamma') \le R. $$
\end{definition}

Note that when we say $Y$ is stable in $X$ we mean that $Y$ is stable in $X$ with respect to a particular quasi-isometric embedding $Y \to X$. Such a quasi-isometric embedding will always be clear from context, e.g., an undistorted subgroup $H$ of a finitely generated group $G$.

\begin{remark}\label{stab implies qc}
The condition that $f: Y \to X$ is a $K$-quasi-isometric embedding for some $K \ge 1$ implies that  if $\gamma$ is an $L$-quasigeodesic that begins and ends on the image of $Y$ then it remains within an $R'$-neighborhood of $f(Y)$ where $R' = R(\max\{K,L \})$. In particular, $f(Y)$ is quasiconvex in $X$. To see these statements, note that if we let $\sigma$ be a geodesic in $Y$ whose end points map under $f$ to the end points of $\gamma$, then $f(\sigma)$ is a $K$-quasigeodesic and $d_{\mathrm{Haus}}(\gamma,\sigma) \le R'$.
\end{remark}

It is well-known that when $X$ is $\delta$-hyperbolic, the preimage of a quasiconvex subspace through a quasi-isometric embedding is itself quasiconvex. This property, however, fails when the space $X$ is not hyperbolic. An important property of stability is that it is preserved under quasi-isometric embeddings. This will be especially important when characterizing stable subgroups of mapping class groups.

\begin{proposition}\label{preserved}
Suppose that $X,Y,Z$ are geodesic metric spaces and $X \to Y \to Z$ are quasi-isometric embeddings. If $X$ is stable in $Z$, then $X$ is stable in $Y$.
\end{proposition}

\begin{proof}
Let $\gamma_1$ and $\gamma_2$ be $L$-quasigeodesics in $Y$ which share endpoints in $X$. If $f:Y \to Z$ is a $K$-quasi-isometric embedding, then $f(\gamma_1), f(\gamma_2)$ are $L'$-quasigeodesics in $Z$ that share endpoints in $X$, where $L'$ depends only on $L$ and $K$. By stability of $X$ in $Z$, these quasigeodesics remain within an $R$-neighborhood of one another, for $R$ depending on $L'$. We conclude that $\gamma_1$ and $\gamma_2$ have Hausdorff distance no greater than $K(R+K)$. Since these constants depend only on $K$ and $L$, this completes the proof.
\end{proof}

\begin{lemma}
If $Y$ is stable in $X$ then $Y$ is $\delta$-hyperbolic for some $\delta \ge 0$.
\end{lemma}

\begin{proof}
This follows from well-known arguments, so we only provide a sketch. See Lemma $6.2$ of \cite{MM1} for details. Fix a $K$-quasi-isometric embedding $f:Y \to X$. Let $x,y,z \in Y$ and consider geodesics $[x,y],[y,z],[x,z]$ in $Y$ joining these points. It suffices to show that $[x,y]$ is contained in the $\delta$-neighborhood of the other two geodesics, for $\delta$ depending only on $K$ and the stability constants. 

If $z'$ denotes a point on $[x,y]$ nearest to $z$ in $Y$, then both $[z,z'] \cup [z',x]$ and $[z,z'] \cup [z',y]$ are $3$-quasigeodesics, where $[x,y] = [x,z'] \cup [z',y]$. The point is that the images of these (quasi-) geodesics under $f:Y \to X$ are quasigeodesics with uniform constants. Hence, there is an $R \ge 0$ depending only on these constants so that $f([z,z'] \cup [z',x]) \subset N_R\left(f([x,z])\right)$ and $f([z,z'] \cup [z',y]) \subset N_R\left(f([y,z])\right)$. Since $f$ is a $K$-quasi-isometric embedding, this implies that every point on $[x,y]$ is within $K(R+K)$ of some point on either $[x,z]$ or $[y,z]$. This completes the proof. 
\end{proof}

Although we have defined stability in a general setting, our focus will be the case of a finitely generated group $G$. Fix a finite generating set $S$ of $G$ and let $|\cdot|_S$ be the associated word metric. Recall that any two generating sets of $G$ give quasi-isometric metrics and that a finitely generated subgroup $H \le G$ is called \emph{undistorted} if the inclusion $H \to G$ is a quasi-isometric embedding for some (any) word metrics on $H$ and $G$. 

\begin{definition}
Let $G$ be a finitely generated group with word metric $|\cdot|_S$. Then $H \le G$ is \emph{stable} if $H$ is undistorted in $G$ and $H \subset (G, |\cdot|_S)$ is stable (as in Definition \ref{def: stab})  for any choice of word metric on $H$.
\end{definition}

Note that in the definition of stability for $H \le G$, since $H$ is undistorted in $G$ one can use any word metric on $H$ when defining stability. The next lemma, whose proof follows directly from Lemma \ref{preserved}, shows that the stability of $H \le G$ is also independent of the word metric on $G$.

\begin{lemma}
Let $G$ be a finitely generated group. If $H \le G$ is stable with respect to some word metric $| \cdot|_S$ on $G$, then it is stable with respect to any word metric on $G$.
\end{lemma}

\begin{remark}
For a finitely generated group $G$, the property of stability of a subgroup $H$ is well studied in the case where $H$ is cyclic. In this case, $H = \langle h \rangle$, the generating element $h$ is usually called either \emph{stable} or \emph{Morse}. See \cite{DMS} and the references found there.
\end{remark}

\section{The Masur-Minsky criteria for stability}
To show that convex cocompact subgroups of mapping class groups are stable, we use the criterion for hyperbolicity developed by Masur-Minsky in \cite{MM1} which we adapt for our purposes. 

We say that a family of paths $\Gamma$ in $X$ is \emph{transitive} for a subspace $Y \subset X$ if any two points in $Y$ can be connected by a path in $\Gamma$.

\begin{definition}\label{contracting}
Let $X$ be a metric space with subspace $Y\subset X$.  Let $\Gamma$ be a transitive family of paths in $X$ between points in $Y$ with the following property: for each $\beta: I \to X$ there exists a map $\pi: X \to I$ and constants $a,b,c >0$ such that

\begin{enumerate}
\item For any $t \in I$, $\mathrm{diam}_X(\beta([t,\pi(\beta(t))])) \le c$. 
\item If $d(x_1,x_2) \le 1$, then $\mathrm{diam}(\beta([\pi(x),\pi(y)])\le c$. 
\item If $d(x, \beta(\pi(x))) \ge a$ and $d(x,x') \le b \cdot d_X(x, \beta(\pi(x)))$, then
$$\mathrm{diam}(\beta[\pi(x),\pi(x')] )\le c. $$
\end{enumerate}
Then we call $\Gamma$ a \emph{family of uniformly contracting paths for $Y\subset X$}.
\end{definition}

\begin{remark}
We do not assume that the projection $\beta(\pi(x))\in \mathrm{im}(\beta)$ is a bounded distance from the points on $\beta$ which are closest to $x$.
\end{remark}

\begin{proposition}[\cite{MM1}]\label{criteria}
Let $\Gamma$ be a family of uniform contracting paths for $Y \subset X$. Then for any $L\ge 0$ there exists an $R= R(L) \ge 0$ with the following property: if $\gamma$ is a $L$-quasigeodesic that begins and ends on $Y$, then for any $\beta \in \Gamma$ with $\beta$ having the same endpoints of $\gamma$, we have $d_{\mathrm{Haus}}(\gamma, \beta) \le R$. \end{proposition}

\begin{proof}
This is proven in Lemma $6.1$ of \cite{MM1}, which states that any space $X$ with a family of uniformly contracting paths for $X$ has stability of quasigeodesics. The proof of Lemma $6.1$ shows that any $L$-quasigeodesic $\alpha$ whose endpoints agree with a $\beta \in \Gamma$ is contained in an $R'$-neighborhood of $\beta$, where $R'$ depending only on $L$ and the uniform contracting constants ($a,b,c$ from Definition \ref{contracting}). This proves our proposition.
\end{proof}

The following consequence is now immediate:

\begin{corollary} \label{criteria cor}
Suppose that $f: Y \to X$ is a quasi-isometric embedding between geodesic metric spaces and that $f(Y) \subset X$ has a family of uniform contracting paths. Then $Y$ is stable in $X$.
\end{corollary}

\section{Convex cocompactness implies stability}\label{ccimplies}

Our starting point is the following characterization of convex cocompact subgroups of the mapping class group, which follows easily from \cite{KL2} or by combining results of \cite{FMo} and \cite{Rafihyp}. We provide a few details using these references.

Recall that a collection of markings $M \subset \mathcal{M}(S)$ is called \emph{$E$-cobounded} if for any $\mu, \nu \in M$ and any proper subsurface $Y \subsetneq S$, we have $d_Y(\mu, \nu) \le E$.  
\begin{lemma} \label{cobounded}
Let $G$ be a finitely generated subgroup of $\Mod(S)$. If $G$ is convex cocompact, then, for any marking $\mu \in \mathcal{M}(S)$, there is an $E \ge 0$ so that the orbit $G \cdot \mu \subset \M(S)$ is $E$-cobounded. Conversely, if $G$ is undistorted in $\Mod(S)$ and there is a marking $\mu \in \mathcal{M}(S)$ and an $E\ge 0$ so that $G \cdot \mu$ is $E$-cobounded, then $G$ is convex cocompact. 
\end{lemma}

\begin{proof}
The first statement is contained in the proof of Theorem $7.4$ of \cite{KL2}, where the assumption on $G$ is that the orbit map from $G$ into $\C(S)$ is a quasi-isometric embedding.

Alternatively, we can see $E$-coboundedness of the orbit $G \cdot \mu$ using the fact that orbits of $G$ in $\T(S)$ are quasiconvex.  Let $x \in \T(S)$ be such that every curve in $\mu$ has bounded length in $x$. If there exists subsurfaces $Y_i \subsetneq S$ and $g_i \in G$ with
$$d_{Y_i}(\mu, g_i\cdot  \mu) \to \infty$$
then the Teichm\"{u}ller geodesics $\tau_i$ joining $x$ and $g_i\cdot x$ become $\epsilon_i$-thin for $\epsilon_i \to 0$ [Theorem 5.5, \cite{Rafihyp}]. (See also Theorem 4.1 of \cite{RScovers}.)  However, the orbit $G \cdot x$ is in some fixed thick part of $\T(S)$ and so we must have that points along $\tau$ get arbitrary far from the orbit $G \cdot x$. This contradicts orbit quasiconvexity of $G$ in $\T(S)$.

The second statement follows from the Masur-Minsky distance formula (Theorem \ref{distance}). Since $G$ is undistorted, we may coarsely measure distance in $G$ by distance in the orbit $G \cdot \mu \subset \mathcal{M}(S)$. That is, for any $g_1,g_2 \in G$
\begin{eqnarray} \label{MMform}
d_G(g_1,g_1)  \asymp d_{\M(S)}(g_1 \cdot \mu, g_2 \cdot \mu) \asymp d_{S}(g_1\cdot \mu,g_2 \cdot \mu) + \sum_{Y \subsetneq S}[d_Y(g_1\cdot \mu,g_2\cdot \mu)]_A,
\end{eqnarray}
where the symbol $\asymp$ depends only on the surface $S$ and the quasi-isometry constant of the orbit map $G \to \M(S)$. Choosing the threshold $A$ in the distance formula (Theorem \ref{distance}) to be larger than $E$ shows that distance in $G$ is coarsely distance in its curve graph orbit. Hence, the orbit map $G \to \C(S)$, given by $g \mapsto  g \cdot \mu$, is a quasi-isometric embedding and so $G$ is convex cocompact by Theorem \ref{cc}. This completes the proof.
\end{proof}

For the remained of this section suppose that $G \le \Mod(S)$ is convex cocompact.  In order to show that $G$ is stable in $\Mod(S)$, it suffices to show that $G \cdot \mu$ is stable in $\M(S)$. This follows from the fact that $\Mod(S)$ is quasi-isometric to $\M(S)$ and $G$ is undistorted in $\Mod(S)$. By Lemma \ref{cobounded}, $G\cdot \mu$ is $\widebar{E}$-cobounded for some $\widebar{E} \ge 0$. We proceed by showing that hierarchy paths form a family of uniform contracting paths for $G \cdot \mu \subset \M(S)$.

Our use of hierarchy paths is motivated by the Slice Comparison Lemma [Lemma 6.7, \cite{MM2}], Behrstock's work on the asymptotic cone of $\Mod(S)$ \cite{Be}, and more recent work of Sisto \cite{Sistocont}. To provide the most direct proof of Theorem \ref{ccimples}, we have chosen to use a theorem of Duchin-Rafi \cite{DuRa} (stated as Theorem \ref{DR} below), which is compatible with the sort of projection from Definition \ref{contracting}.  The strong contraction property of hierarchy paths between cobounded pants decompositions was also proven in Theorem $4.4$ of \cite{BMMII} in their work on the Weil-Petersson geometry of $\T(S)$.  We have included a proof of Proposition \ref{hc} here for completeness and as an application of the Masur-Minsky criteria (Theorem \ref{criteria}).\\ 

Let $p: \M(S) \to \C(S)$ be the map which associates to a marking $\mu$ the collection of curves which appear in the base  of $\mu$, i.e $a \in p(\mu)$ if and only if $a \in \mathrm{base}(\mu)$. This map, called the shadow map, is coarsely $4$-Lipschitz [Lemma 2.5, \cite{MM2}]. For any markings $\mu ,\nu \in \M(S)$, let 
$$H= H(\mu,\nu): [0,N] \to \M(S)$$
be a hierarchy path with $H(0) = \mu$ and $H(N) = \nu$. Recall that $H$ is a $M$-quasigeodesic in $\M(S)$, where M depends only on the topology of $S$ (Theorem \ref{hier qg} $(1)$). If $\mu$ and $\nu$ are $\widebar{E}$-cobounded, then all the markings that appear in $H$ are $E$-cobounded, for some $E \ge 0$ that depends only $\widebar{E}$ and the surface $S$ (Theorem \ref{hier qg} $(4)$).
By the argument in Lemma \ref{cobounded}, $h = p \circ H: [0,N] \to \C(S)$ is a quasigeodesic in $\C(S)$ from $p(\mu)$ to $p(\nu)$ whose quasigeodesic constant depends only on $E$. More precisely, if we choose the cut off $A$ in the distance formula (Theorem \ref{distance}) to be larger than $E$, then for any $i,j \in [0,N]$
\begin{eqnarray*}
|i-j| &\le& M\cdot  d_{\M(S)}(H(i),H(j)) +M \\
&\le& M K \cdot d_{S}(h(i),h(j)) + MK^2 +M,
\end{eqnarray*}
where $K$ depends only on $E$ and $S$. Hence, $h = p \circ H$ is a $K_E$-quasigeodesic in $\C(S)$, where $K_E = K^2M$.

As $h:[0,N] \to \C(S)$ is a $K_E$-quasigeodesic into the $\delta$- hyperbolic space $\C(S)$, there is a nearest point retraction $\bold{n}_h: \C(S) \to h$ as discussed in Subsection \ref{hypgeo}. Define the projection $\hpr_h: \C(S) \to [0,N]$ to the domain of the path $h$ so that for $c \in \C^0(S)$, 
$$h(\hpr_h(c)) = \bold{n}_h(c).$$ 
That is, $\hpr_h(c)$ is a parameter $i \in [0,N]$ so that the distance from $c$ to the image of $h$ is minimized at $h(i)$. By the properties of $\bold{n}_h$ stated in Subsection \ref{hypgeo}, it is immediate that there is an $L$ (depending only on $E$) so that this projection is both $L$-coarsely well-defined and coarsely $L$-Lipschitz. We emphasize that this uses only the facts that $\C(S)$ is hyperbolic and $h:[0,N] \to \C(S)$ is a quasigeodesic.

The projection $\hpr_h: [0,N] \to \C(S)$ induces a corresponding map from $\M(S)$ to $[0,N]$. Let
 $$\pr_H: \M(S) \to [0,N]$$ 
be defined as follows: for any $\alpha \in \M(S)$, set $\pr_H(\alpha) = \hpr_h(a)$, for some choice of curve $a \in p(\alpha) \subset \C(S)$. Note that for different choices of curves $a,a'  \in p(\alpha)$ we have $d_S(a,a') \le 4$ and so $|\hpr_h(a) - \hpr_h(a')| \le 4L$. 
 
 \begin{remark}
 It may seem slightly unnatural to define the projection $\pr_H$ to the domain of the path $H$, rather than to its image in $\M(S)$. We have done so for two reasons. First, this allows for a direct application of Theorem \ref{criteria}, which verifies that quasigeodesics fellow travel in a uniform way. Indeed, projecting to the domain of a path is the approach of Masur-Minsky in \cite{MM1}. Second, it seems that such a projection of a marking $\mu$ to a hierarchy path $H$ need not be a uniformly bounded distance from the closet point to $\mu$ on $H$. Using Theorem \ref{criteria} avoids this subtlety. 
 \end{remark}

Our goal for the rest of this section, achieved in Proposition \ref{hc} below, is to show that the collection of hierarchy paths between markings in a fixed orbit of a convex cocompact subgroup $G$ is a family of uniform contracting paths for the orbit in $\M(S)$. That convex cocompact subgroups are stable, Theorem \ref{ccimples} below, follows quickly from Proposition \ref{hc} and Corollary \ref{criteria cor}. The proof of the contracting property uses the following theorem of Duchin and Rafi \cite{DuRa} (see also Theorem $4.3$ of \cite{BMMII}):

\begin{theorem}[Theorem 4.2, \cite{DuRa}] \label{DR}
Given $E$ there exist $B_1$ and $B_2$ so that if $H$ is a hierarchy path in $\mathcal{M(S)}$ between $E$-cobounded markings $\mu$ and $\nu$, then for any $\alpha \in \mathcal{M}(S)$ with $d(\alpha, H) \ge B_1$ and $R = d(\alpha, H) / B_1$, we have
$$\mathrm{diam}_{\mathcal{C}(S)}(h(\mathrm{Proj}_{H}(B_R(\alpha)))) \le B_2,$$
where $B_R$ denote the $R$-ball in $\M(S)$.
\end{theorem}

\begin{remark}
In the statement of Theorem \ref{DR} in \cite{DuRa}, the authors allow any quasigeodesic in $\M(S)$ all of whose markings are uniformly cobounded. This is automatically satisfied by the hierarchy path $H$ (Theorem \ref{hier qg} $(4)$).
\end{remark}

\begin{proposition}\label{hc}
Let $M \subset \M(S)$ be a collection of $E$-cobounded markings. The set of all hierarchy paths between markings in $M$ is a family of uniformly contracting paths for $M \subset \M(S)$. 
\end{proposition}

\begin{proof}
Let $\mu, \nu \in M$ be arbitrary and let $H:[0,N] \to \M(S)$ be an hierarchy path with $H(0) =\mu$ and $H(N) = \nu$. We show that the conditions from Proposition \ref{criteria} are satisfied for the projections $\pr_H$ defined above, with constants $a,b,c$ depending only on $E$. For $(1)$, we must show that for any $i \in I$, 
$$\mathrm{diam}_{\M(S)}H([i,\pr_H(H(i))])$$
is bounded by a constant depending only on $K_E$. As $H$ is an $M$-quasigeodesic, this quantity is bounded by $M |i - \pr_H(H(i))| +M$. Since $h(i) = p(H(i))$, for any curve $a \in p(H(i))$ the difference $|\hpr_h(a) - i|$ is bounded by $L$. Hence,
$$M |i - \pr_H(H(i))| +M \le M |i - \pr_h(a)| +M  \le M\cdot L +M,$$
as required.

For $(2)$, we show that for any $\alpha, \beta \in \M(S)$ with $d_{\M(S)}(\alpha,\beta) \le 1$
$$\mathrm{diam}_{\M(S)}(H([\pr_H(\alpha), \pr_H(\beta)])) $$
is bounded by a constant depending only on $E$. This is similar to $(1)$, since again it suffices to bound $|\pr_H(\alpha) -  \pr_H(\beta)|$. Let $a \in p(\alpha)$ and $b \in p(\beta)$, then $d_S(a,b) \le 8$ ($p$ is 4-Lipschitz) and so $|\hpr_h(a) - \hpr_h(b)| \le 8L$. Hence,
$$|\pr_H(\alpha) -  \pr_H(\beta)| \le |\hpr_h(a) - \hpr_h(b)| \le 8L.$$

For $(3)$, we will apply Theorem \ref{DR}. Before doing so, we must first show that there is a $B_3 \ge 0$ so that for any $\alpha \in \M(S)$, 
$$ d_{\M(S)}(\alpha, H(\pr_H(\alpha)) \le B_3 \cdot d_{\M(S)}(\alpha, H), $$
whenever $d_{\M(S)}(\alpha, H)$ is sufficiently large. Here, as in Theorem \ref{DR}, $d_{\M(S)}(\alpha, H)$ is the minimum distance from $\alpha$ to any marking in the image of $H$.

Let $n(\alpha)$ be a marking on $H$ that is closest to $\alpha$ and set $\bar{\alpha} = H(\pr_H(\alpha))$. By construction, $d_S(\alpha, \bar{\alpha}) \le d_S(\alpha, n(\alpha))$. For a proper subsurface $Y \subsetneq S$,
\begin{eqnarray}\label{nearest}
d_Y(\alpha, \bar{\alpha}) \le d_Y(\alpha, n(\alpha)) + d_Y(n(\alpha),\bar{\alpha}) \le d_Y(\alpha,n(\alpha)) + E 
\end{eqnarray}

Plugging the above inequality (\ref{nearest}) into the distance formula (Theorem \ref{distance}) with threshold $A \ge 2E$ gives
$$d_{\M(S)}(\alpha, \bar{\alpha}) \le B_3 \cdot d_{\M(S)}(\alpha, n(\alpha)), $$
for some $B_3 \ge 0$. Note that to eliminate the additive constant in the distance formula, we have used that $d_{\M(S)}(\alpha, n(\alpha))\neq 0$.

Now set $a = B_1B_3, b = \frac{1}{B_1B_3},$ and $c =M(K_E \cdot B_2 +K_E +1)$, where $B_1$ and $B_2$ are as in Theorem \ref{DR} and $B_3$ was determined above. Let $\alpha \in \M(S)$ with $d_{\M(S)}(\alpha, H(\pr_H(\alpha)))\ge a$ and $\beta \in \M(S)$ with $d_{\M(S)}(\alpha,\beta) \le b \cdot d_{\M(S)}(\alpha, H(\pr_H(\alpha)))$. Then 
$$d_{\M(S)}(\alpha, H)\ge 1/B_3 \cdot d_{\M(S)}(\alpha, H(\pr_H(\alpha)))\ge a/B_3 = B_1$$
and
\begin{eqnarray*}
d_{\M(S)}(\alpha,\beta) &\le& b \cdot d_{\M(S)}(\alpha, H(\pr_H(\alpha))) \\
&\le& \frac{1}{B_1B_3} \cdot d_{\M(S)}(\alpha, H(\pr_H(\alpha))) \\
&\le& \frac{1}{B_1} \cdot d_{\M(S)}(\alpha, H). 
\end{eqnarray*}
Hence, $\beta \in B_R(\alpha)$ for $R = \frac{1}{B_1}d_{\M(S)}(\alpha, H)$, and so by Theorem \ref{DR}
$$d_{\C(S)}(h(\pr_H(\alpha)),h(\pr_H(\beta))) \le B_2. $$

Since $h$ is a $K_E$-quasigeodesic, $|\pr_H(\alpha) - \pr_H(\beta)| \le K_E \cdot B_2 +K_E$, and so we conclude that 
 $$\mathrm{diam}_{\M(S)}(H([\pr_H(\alpha),\pr_H(\beta)]) \le M(K_E \cdot B_2 +K_E) +M = c.$$
This completes the proof of condition $(3)$, and shows that the collection of hierarchy paths between markings in $M$ is a family of uniformly contracting paths for $M$ in $\M(S)$.
\end{proof}

\begin{theorem} \label{ccimples}
If $G \le \Mod(S)$ is convex cocompact, then $G$ is stable.
\end{theorem}

\begin{proof}
Since $G$ is convex cocompact, it is undistorted in $\Mod(S)$. Fix $\mu \in \M(S)$ and recall that the orbit map $\Mod(S) \to \M(S)$, given by $g \mapsto g \cdot \mu$, is a quasi-isometry.  Hence, by Proposition \ref{preserved}, it suffices to show that $G \to \M(S)$ is stable. By Theorem \ref{cobounded}, the orbit $G \cdot \mu$ is $E$-cobounded for some $E\ge 0$. Proposition \ref{hc} then implies that the set of all hierarchy paths between vertices in the orbit $G \cdot \mu$ is a family of uniform contracting paths for $G\cdot \mu$ in $\M(S)$. By Corollary \ref{criteria cor}, this implies that $G \to G\cdot \mu \subset \M(S)$ is stable. Thus $G$ is a stable subgroup of $\Mod(S)$.
\end{proof}

The following corollary is now immediate.
\begin{corollary}\label{anygen}
Let $G \le \Mod(S)$ be convex cocompact. Then $G$ is quasiconvex in $\Mod(S)$ with respect to any (finite) generating set.
\end{corollary}

\section{Stability implies convex cocompactness} \label{stabilityimples}

The proof of the converse to Theorem \ref{ccimples} is a straightforward contradiction argument using the structure of hierarchy paths and the marking complex (see Section \ref{Modbackground}).  Before we proceed with the proof, we recall some notions from \cite{BMin} about the product regions of $\M(S)$ associated to simplices of $\C(S)$. By a simplex of $\C(S)$, we mean a collection of pairwise adjacent vertices of $\C(S)$. 

Let $\Delta \subset \C(S)$ be a simplex, and let $\Q(\Delta) \subset \M(S)$ be the set of markings whose bases contain $\Delta$.  We note that $\Q(\Delta)$ is quasi-isometric to $\mathrm{Stab}_{\Mod(S)}(\Delta)$, the stabilizer of $\Delta$ in $\Mod(S)$.  Let $\sigma(\Delta) \subset S$ be the components of $S\setminus \Delta$ which are not pairs of pants, including the annuli about the curves of $\Delta$. There is a map  
$$\Q(\Delta) \to \prod_{Y \in \sigma(\Delta)} \M(Y)$$
where $\M(Y)= \C(Y)$ if $Y$ is an annulus.  The map is given by restricting (or projecting) a marking $\mu \in \Q(\Delta)$ to markings on the subsurfaces in $\sigma(\Delta)$ and, for each $\alpha \in \Delta$, associating the transversal to $\alpha$ in $\mu$ to a corresponding point in $\C(\alpha)$.  The following lemma is essentially an application of the distance formula (Theorem \ref{distance}):

\begin{lemma}[Lemma 2.1, \cite{BMin}] \label{BM prod}
The correspondence
$$ \Q(\Delta) \to \prod_{Y \in \sigma(\Delta)} \M(Y)$$
is a $P$-quasi-isometry, where $P\ge 0$ depends only on the surface $S$.
\end{lemma}

We can now give the idea of the proof of Theorem \ref{stabimplies}. If a group $G \le \Mod(S)$ is stable but not convex cocompact, then Lemma \ref{cobounded} implies that the $G$-orbit of some marking $\mu \in \M(S)$ does not have bounded subsurface projections.  Thus, for any $E>0$, we can find a marking $\nu \in G\cdot \mu$ such that $d_Y(\mu, \nu)>E$ for some proper subsurface $Y \subsetneq S$.  Theorem \ref{hier qg} implies that there is a hierarchy path $H$ from $\mu$ to $\nu$ and a subsegment $H_Y \subset H$ with $H_Y \subset \Q(\partial Y)$ so that $\left|H_Y\right| \succ E$.  If $\alpha_Y, \beta_Y \in H_Y$ are the initial and terminal markings of $H_Y$, respectively, then stability of $G$ implies that there are markings $\mu_1, \mu_2 \in G \cdot \mu$ such that $\mu_1$ and $\mu_2$ are within some uniform distance of $\alpha_Y$ and $\beta_Y$.  Using the product structure in Lemma \ref{BM prod}, we can use $H_Y$ to build two quasigeodesics between $\mu_1$ and $\mu_2$ with constants depending only on $S$, whose Hausdorff distance is coarsely at least $E$. Since $E>0$ was chosen arbitrarily, this contradicts the stability assumption for $G \cdot \mu$.\\

In this last step, we are taking advantage of the well-known fact that quasigeodesics in product spaces need not fellow-travel, a variation of which we record in the following lemma:

\begin{lemma} \label{quasiinflats}
Let $\mathcal{X}$ and $\mathcal{Y}$ be connected, infinite-diameter graphs and let $\mathcal{Z}$ be the 1-skeleton of $\mathcal{X} \times \mathcal{Y}$, endowed with the graph metric.  For any vertices $z_1 = (x_1,y_1)$ and $z_2 = (x_2,y_2)$ of $\mathcal{Z}$, there are $3$-quasigeodesics $\gamma_a ,\gamma_b$ each from $z_1$ to $z_2$ so that 
$$d_{\mathrm{Haus}}(\gamma_a,\gamma_b) \ge \max \{d_X(x_1,x_2), d_Y(y_1,y_2) \}.$$
\end{lemma}

\begin{figure}[htbp]
\begin{center}
\includegraphics[width = 100mm]{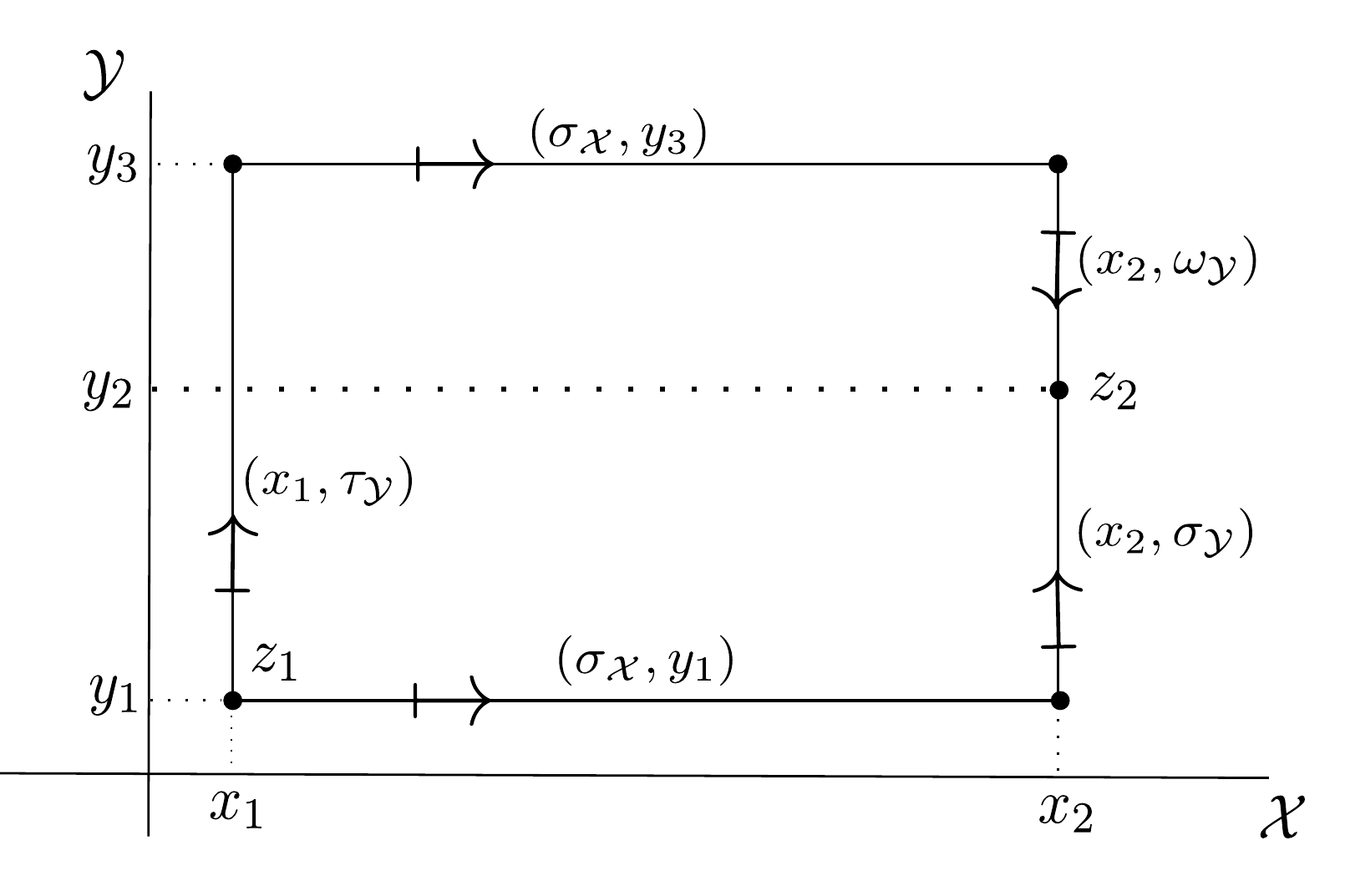}
\caption{The path $\gamma_a=(\sigma_{\mathcal{X}},y_1) \cdot (x_2,\sigma_{\mathcal{Y}})$ is far from the path $\gamma_b= (x_1,\tau_{\mathcal{Y}}) \cdot ( \sigma_{\mathcal{X}}, y_3) \cdot (x_2, \omega_{\mathcal{Y}})$. }
\label{flats}
\end{center}
\end{figure}

\begin{proof}
The proof is easily seen with Figure \ref{flats}, but we provide the written details here. Since we are working only with graphs, all paths will be considered as sequences of adjacent vertices indexed by intervals of integers. Hence, for a path $\gamma:[0,N] \to \mathcal{Z}$, we have $\ell(\gamma|_{[i,j]}) = |j-i|$, where the length of such a path is the number of edges it traverses. In this case, to show that $\gamma$ is a $3$-quasigeodesic, it suffice to show that for any $i \le j$, 
$$\ell(\gamma|_{[i,j]}) \le 3\cdot d_{\mathcal{Z}}(\gamma(i),\gamma(j)). $$
Also, recall that for any two vertices $z_1 = (x_1,y_1)$ and $z_2=(x_2,y_2)$ of $\mathcal{Z}$, the graph metric is
\begin{eqnarray} \label{Zmetric}
d_{\mathcal{Z}}(z_1,z_2) = d_{\mathcal{X}}(x_1,x_2) + d_{\mathcal{Y}}(y_1,y_2). 
\end{eqnarray}

Now, suppose that $d:= d_{\mathcal{X}}(x_1,x_2) \ge d_{\mathcal{Y}}(y_1,y_2)$ and let $\sigma_{\mathcal{X}}$ be a geodesic path in $\mathcal{X}$ that joins $x_1$ to $x_2$. Similarly, let $\sigma_{\mathcal{Y}}$ be a geodesic path in $Y$ that joins $y_1$ to $y_2$. For any $y \in \mathcal{Y}$, we denote by $(\sigma_{\mathcal{X}},y)$ the corresponding geodesic path in $\mathcal{Z}$ whose first coordinate entries are the vertices of $\sigma_{\mathcal {X}}$ and whose second coordinate is $y$. With this notation, let $\gamma_a$ be the path in $\mathcal{Z}$ that is a the concatenation (read from left to right) 
$$\gamma_a = (\sigma_{\mathcal{X}},y_1) \cdot (x_2,\sigma_{\mathcal{Y}}). $$
It is clear that $\gamma_a$ is an path of adjacent vertices of $\mathcal{Z}$ which joins $z_1$ to $z_2$. Moreover, $\gamma_a$ is a geodesic path. This follows from Equation \ref{Zmetric} and the observation that $\gamma_a$ does not backtrack in either coordinate.

We now construct a $3$-quasigeodesic $\gamma_b$ which also joins $z_1$ to $z_2$ but travels far from $\gamma_a$. Let $y_3$ be a vertex of $\mathcal{Y}$ with the property that $d_{\mathcal{Y}}(y_1,y_3) = d_{\mathcal{X}}(x_1,x_2) = d$. Let $\tau_{\mathcal{Y}}$ be a geodesic path in $\mathcal{Y}$ between $y_1$ and $y_3$ and let $\omega_{\mathcal{Y}}$ be a geodesic path in $\mathcal{Y}$ between $y_3$ and $y_2$. Now define $\gamma_b$ to be the path of adjacent vertices of $\mathcal{Z}$ given by
$$\gamma_b = (x_1,\tau_{\mathcal{Y}}) \cdot ( \sigma_{\mathcal{X}}, y_3) \cdot (x_2, \omega_{\mathcal{Y}}).$$
As before, $\gamma_b$ is a path from $z_1$ to $z_2$. Also, since $\gamma_b$ contains the point $(x_1,y_3)$, the Hausdorff distance from $\gamma_b$ to $\gamma_a$ is no less than the distance from $(x_1,y_3)$ to $\gamma_a$. Since,
\begin{eqnarray*}
d_{\mathcal{Z}}((x_1,y_3), \gamma_a) &=& \min\{ d_{\mathcal{Z}}((x_1,y_3),(\sigma_{\mathcal{X}},y_1)), \; \;  d_{\mathcal{Z}}((x_1,y_3),   (x_2,\sigma_{\mathcal{Y}})  \} \\
&\ge& \min\{ d_{\mathcal{Z}}((x_1,y_3),(x_1,y_1)), \; \;  d_{\mathcal{Z}}((x_1,y_3),   (x_2, y_3) ) \} \\
&=& \min \{d_{\mathcal{Y}}(y_3,y_1), \; \; d_{\mathcal{X}}(x_1,x_2)  \}  = d,
\end{eqnarray*}
we have $d_{\mathrm{Haus}(\mathcal{Z})}(\gamma_a,\gamma_b) \ge d = d_{\mathcal{X}}(x_1,x_2)$. It remains to show that $\gamma_b$ is a $3$-quasigeodesic.

As in the construction of $\gamma_a$, both $(x_1,\tau_{\mathcal{Y}}) \cdot ( \sigma_{\mathcal{X}}, y_3)$ and $( \sigma_{\mathcal{X}}, y_3) \cdot (x_2, \omega_{\mathcal{Y}})$ are geodesic subpaths of $\gamma_b$. Hence, let $z_i \in (x_1,\tau_{\mathcal{Y}})$ and let $z_j \in (x_2, \omega_{\mathcal{Y}})$ and note that $d_{\mathcal{Z}}(z_i,z_j) \ge d_{\mathcal{X}}(x_1,x_2)$. Denote by $\gamma_b|[z_i,z_j]$ the portion of $\gamma_b$ between $z_i$ and $z_j$. We compute
\begin{eqnarray*}
\ell(\gamma_b|[z_i,z_j]) &=&  d_{\mathcal{Z}}(z_i, (x_1,y_3)) + d_{\mathcal{Z}}((x_1,y_3),(x_2,y_3)) + d_{\mathcal{Z}}((x_2,y_3), z_j) \\
& \le& 3 \cdot d_{\mathcal{X}}(x_1,x_2) \\
& \le& 3 \cdot d_{\mathcal{Z}}(z_i,z_j).
\end{eqnarray*}
By our remark in the first paragraph of this proof, we are done.
\end{proof}

The following theorem completes the proof of Theorem \ref{intro_main}.
\begin{theorem}\label{stabimplies}
Suppose that $G \le \Mod(S)$ is stable. Then $G$ is convex cocompact.
\end{theorem}

\begin{proof}
Assume towards a contradiction that $G \le \Mod(S)$ is stable but not convex cocompact. For a fixed $\mu \in \M(S)$, the orbit map $G \to \M(S)$ given by $g \mapsto g \cdot \mu$ is a $K$-quasi-isometric embedding for some $K\ge 1$. By Proposition \ref{preserved}, we have that $G\cdot \mu$ is stable in $\M(S)$.

Let $L = \max\{K,M\} $, where $M$ is the quasigeodesic constant for a hierarchy path (Theorem \ref{hier qg} (\ref{hier qg 1})), and set $A_1 = R(L)$, the stability constant for $L$-quasigeodesics in $\M(S)$ which begin and end on $G \cdot \mu$. Since any two markings in the orbit $G \cdot \mu$ are joined by both hierarchy paths and $K$-quasigeodesics which are contained in $G \cdot \mu$, any hierarchy path between markings in $G \cdot \mu$ is contained in the $A_1$-neighborhood of $G \cdot \mu$ (see Remark \ref{stab implies qc}).  Finally, set $A_2 = R(3P^2+2A_1)$, where $P$ is as in Theorem \ref{BM prod}. 

Since $G$ is undistorted but not convex cocompact, Proposition \ref{cobounded} implies that for any $E \ge 0$ there is a $g \in G$ and a proper subsurface $Y \subset S$ such that 
\begin{eqnarray*}
d_{Y}(\mu, g \cdot \mu) \ge E.
\end{eqnarray*}
For $E \ge M_1$, Theorem \ref{hier qg} (\ref{hier qg 3}) implies that any hierarchy path $H: [0,N] \to \M(S)$ with $H(0) =\mu$ and $H(N)=g \cdot \mu$ contains a subpath $H_Y$ such that $\partial Y \subset \mathrm{base}(\mu)$ for each $\mu \in H_Y$, i.e. $H_Y \subset \Q(\partial Y)$.  If we denote the initial marking and terminal markings of $H_Y$ by $\alpha_Y$ and $\beta_Y$, respectively, then
\begin{eqnarray}\label{Ymarkingdistance}
d_{\M(Y)}(\alpha_Y,\beta_Y) \ge \frac{1}{4} d_{Y}(\alpha_Y,\beta_Y) \ge \frac{1}{4}(E - 2M_2)
\end{eqnarray}
with the last inequality in (\ref{Ymarkingdistance}) following from Theorem \ref{hier qg} (\ref{hier qg 3}).  Set $E' = \frac{1}{4}(E - 2M_2)$.

By Lemma \ref{BM prod}, $\Q(\partial Y)$ is $P$-quasi-isometric to $\prod_{X \subset \sigma(\partial Y)} \M(X)$. To apply Lemma \ref{quasiinflats}, set $\mathcal{Y} = \M(Y)$ and $\mathcal{X} = \prod_{X \subset \sigma(\partial Y)\setminus \{Y\}} \M(X)$. Then 
$$\mathcal{Z} = \mathcal{X} \times \mathcal{Y} =  \prod_{X \subset \sigma(\partial Y)} \M(X).$$

Since $\alpha_Y, \beta_Y \in \Q(\partial Y)$, we may use the correspondence of Lemma \ref{BM prod} to view $\alpha_Y$ and $\beta_Y$ in the product space $Z$. Equation (\ref{Ymarkingdistance}) implies that $d_{\mathcal{Y}}(\alpha_Y, \beta_Y) \ge E'$, where $d_{\mathcal{Y}}(\alpha_Y, \beta_Y)$ is just the distance between $\alpha_Y$ and $\beta_Y$ in $\M(Y)$. Lemma \ref{quasiinflats} implies that there exists $3$-quasigeodesics $\gamma^1$ and $\gamma^2$ in $\mathcal{Z}$ that join the markings $\alpha_Y$ and $\beta_Y$ in $\mathcal{Z}$ and whose Hausdorff distance in $\mathcal{Z}$ is greater than or equal to $E'$.

Using the $P$-quasi-isometry in Theorem \ref{BM prod}, we may view $\gamma^1$ and $\gamma^2$ as $3P^2$-quasigeodesics in $\M(S)$ that join the markings $\alpha_Y$ and $\beta_Y$. Measuring Hausdorff distance in $\M(S)$, we have
\begin{eqnarray}\label{bigHaus}
d_{\mathrm{Haus}(\M(S))}(\gamma^1,\gamma^2) \ge \frac{E'-P}{P}. 
\end{eqnarray}

Since the original hierarchy path $H$ joins markings in $G \cdot \mu$, it is contained in an $A_1$-neighborhood of the orbit $G\cdot \mu$. Hence, there are markings $\mu_1, \mu_2 \in G \cdot \mu$ so that 
\begin{eqnarray}
d_{\M(S)}(\alpha_Y, \mu_1) \le A_1  \quad \text{and} \quad  d_{\M(S)}(\beta_Y, \mu_2) \le A_1,
\end{eqnarray}
where $A_1$ depends only on $S$, as above. By appending initial and terminal geodesic segments of length no more than $A_1$ to $\gamma^1$ and $\gamma^2$, we may consider these paths as $(3P^2+2A_1)$-quasigeodesics in $\M(S)$ that join the orbit points $\mu_1, \mu_2 \in G\cdot \mu$. By our choice of $A_2 = R(3P^2+2A_1)$, the stability constant for $(3P^2+2A_1)$, any two $(3P^2+2A_1)$-quasigeodesics between markings in $G \cdot \mu$ have Hausdorff distance no greater than $A_2$.

Since $E \ge 0$ was arbitrary, we may choose $E > 4(PA_2 +P)+2M_2$ so that 
$$E' = \frac{1}{4}(E - 2M_2) > PA_2 +P.$$
Then by Equation \ref{bigHaus}
\begin{eqnarray*}
d_{\mathrm{Haus}(\M(S))}(\gamma^1,\gamma^2) \ge \frac{E'-P}{P} > A_2.
\end{eqnarray*}

This contradicts the assertion that the $(3P^2+2A_1)$-quasigeodesics $\gamma^1$ and $\gamma^2$ between $\mu_1 ,\mu_2 \in G \cdot \mu$ must be $A_2$-Hausdorff close. We conclude that stable subgroups of $\Mod(S)$ are convex cocompact, as required.
\end{proof}

\bibliographystyle{amsalpha}
\bibliography{StabMap.bbl}
\end{document}